\documentclass[a4paper,12pt]{article}
\usepackage{amsthm}
\usepackage{amsmath}
\usepackage{amssymb}
\usepackage{amsfonts}
\usepackage{graphicx} % Required for inserting images
\usepackage{subfig}
\usepackage{geometry}
\usepackage{authblk}
\usepackage{multicol}

\newtheorem{definition}{Definition}
\newtheorem{remark}{Remark}
\newtheorem{lemma}{Lemma}
\newtheorem{theorem}{Theorem}
\newtheorem{theorem(Perfect Graph Theorem)}{Theorem (Perfect Graph Theorem)}

\newtheorem{corollary}{Corollary}
\newtheorem{conjecture}{Conjecture}

\title{Relative Fractional Packing Number and Its Properties\footnote{This research was conducted during a summer research internship at the Chinese University of Hong Kong.}}

\author{
 Mehrshad Taziki\footnote{mehrshad.taziki@sharif.edu} \\ Sharif Univeristy of Technology
}
\date{}

\begin{document}
\maketitle

\begin{abstract}
The concept of the \textit{relative fractional packing number} between two graphs $G$ and $H$, initially introduced in \cite{alipour2023relative}, serves as an upper bound for the ratio of the zero-error Shannon capacity of these graphs. Defined as:
\begin{equation*}
\sup\limits_{W} \frac{\alpha(G \boxtimes W)}{\alpha(H \boxtimes W)}
\end{equation*}
where the supremum is computed over all arbitrary graphs and $\boxtimes$ denotes the strong product of graphs.

This article delves into various critical theorems regarding the computation of this number. Specifically, we address its $\mathcal{NP}$-hardness and the complexity of approximating it. Furthermore, we develop a conjecture for necessary and sufficient conditions for this number to be less than one. We also validate this conjecture for specific graph families. Additionally, we present miscellaneous concepts and introduce a generalized version of the independence number that gives insights that could significantly contribute to the study of the relative fractional packing number.
\end{abstract}

\section{Introduction}
The problem of determining the Shannon zero-error capacity of a graph remains as a foundational challenge in Information Theory, boasting diverse applications. One notable application involves computing the zero-error capacity of a noisy channel which is represented by a graph \cite{Krner1998ZeroErrorIT}. A significant relation exists between this capacity and the independence number of a graph. Specifically, it is situated between the independence number and the fractional packing number (also known as the fractional independence number). Despite its fundamental importance, calculating this graph property remains a challenging task.

In their work \cite{alipour2023relative}, Alipour and Gohari introduced the concept of the relative fractional packing number as an innovative approach to address this challenging problem. Rather than directly bounding the zero-error capacity, they introduced a quantity that, for a pair of graphs $G$ and $H$, provides an upper bound on the ratio of their zero-error Shannon capacities.

Given the recent introduction of this quantity, there are numerous intriguing problems and challenges related to it to explore.

\subsection{Our Contribution}

In this paper, we focus on the fractional packing number and its significance. In our research, we aim to explore the applications of this concept and establish its connections with other graph invariants, especially the Shannon number and the independence number of graphs. Throughout this paper, we present a collection of theorems and properties. Our objective is to provide valuable insights that can contribute to the discovery of new findings related to this graph quantity.

One of the primary objectives of our study was to investigate the computational complexity of computing and approximating the fractional packing number. By drawing parallels between this quantity and the independence number of a graph, we were able to establish that computing the fractional packing number is an $\mathcal{NP}$-Hard problem. Furthermore, we demonstrated that it is also Poly-APX-hard, showing the challenge to develop efficient approximation algorithms.

In our research, we specifically examined the case when both graphs involved in calculating the fractional packing number are cycles, and we were able to calculate the exact value of this quantity in this case. By employing a generalized technique, we derived lower bounds for this quantity in the general case. We then mentioned some tight cases for this lower bound.

Section \ref{sec:3} of this paper is dedicated to the investigation of scenarios where the fractional packing number is below one. We aimed to identify necessary and sufficient conditions for this phenomenon and formulated a conjecture to address this concept. Additionally, we provided evidence for the validity of this conjecture by proving its soundness for well-known families of graphs, including perfect graphs and cycles.

Lastly, we introduced a generalized concept of the independence number and established a set of theorems concerning this quantity. We also highlighted the close relationship between this generalized independence number and the fractional packing number, suggesting the potential utilization of the independence number in future studies of the relative fractional packing number.

\section{Preliminaries}
In this section, we introduce the necessary concepts and background knowledge related to our study.

Let $G$ be a simple graph. We denote the independence number of $G$ by $\alpha(G)$. Computing the independence number is a known $\mathcal{NP}$-Hard problem \cite{book_1975}, and it is even Poly-APX-complete \cite{BAZGAN2005272}. Two sets of vertices, $S$ and $T$, are said to be disconnected if and only if their intersection is empty and there are no edges between them.

A graph $G$ is called vertex-transitive or point-symmetric if and only if for any pair of vertices $u, v \in V(G)$, there exists an automorphism $\pi$ of $G$ such that $\pi(u) = v$. Note that an automorphism for a graph $G$ is just an isomorphism from $G$ to itself.

The strong product of two graphs $G$ and $H$, denoted as $G \boxtimes H$, is a graph product commonly utilized in combinatorics and information theory. The vertex set of $G \boxtimes H$ is the Cartesian product of the vertex sets of $G$ and $H$. Two vertices $(u_1, v_1)$ and $(u_2, v_2)$ in $G \boxtimes H$ are adjacent if and only if any of the following conditions hold:
\begin{itemize}
\item $u_1 u_2 \in E(G)$ and $v_1 v_2 \in E(H)$
\item $u_1 = u_2$ and $v_1 v_2 \in E(H)$
\item $u_1 u_2 \in E(G)$ and $v_1 = v_2$
\end{itemize}
Using this product, the Shannon zero-error capacity of a graph $G$ is defined as $\theta(G)$, which is obtained as the limit of $\sqrt[n]{\alpha(G^{\boxtimes n})}$ as $n$ approaches infinity. The Shannon zero-error capacity plays a crucial role in evaluating the zero-error capacity of a channel represented by the graph $G$ \cite{Krner1998ZeroErrorIT}.

Calculating the Shannon zero-error capacity efficiently remains a challenge, and even for small graphs like $C_7$, the exact value is unknown. Lovász showed that $\theta(C_5) = \sqrt{5}$. To the best of our knowledge, the Lovász number of a graph \cite{lovasz1979shannon} and the Haemers' number of a graph \cite{hamer} are the most commonly used upper bounds for the Shannon capacity of a graph.

The Integer Programming formulation of the independence number of a graph $G$ is as follows:
\begin{equation*}
\begin{array}{lll}
\text{maximize } & \sum\limits_{v \in V} x_{v} \\
\text{subject to } & \sum\limits_{v \in C} x_{v} \leq 1 \, \, \forall \, C \in \mathcal{C} \\
& x \geq 0
\end{array}
\end{equation*}
Here, $\mathcal{C}$ represents the set of all cliques in the graph $G$. The optimal value of the relaxed linear program of the optimization problem above is known as the fractional packing number, denoted by $\alpha^{*}(G)$.

In \cite{hales1973numerical}, Hales demonstrated that:
\[
\alpha^{*}(G) = \sup\limits_{W} \frac{\alpha(G \boxtimes W)}{\alpha(W)}
\]
Here, the supremum is taken over all arbitrary graphs $W$. Additionally, Hales proved that this supremum is indeed a maximum. He also constructed a maximizer for this term.

Motivated by this, Alipour and Gohari \cite{alipour2023relative} introduced the concept of the relative fractional packing number of two graphs. It is defined as follows:
\[
\sup\limits_{W} \frac{\alpha(G \boxtimes W)}{\alpha(H \boxtimes W)}
\]
They provided a linear programming formulation to compute this number and established several properties relating to bounding the ratios of different graph properties. We delve deeper into this topic in Section \ref{sec:3}.

\section{Calculating the Relative Fractional Packing Number}
\label{sec:2}
This section mentions some results regarding the calculation of $\alpha^{*}(G|H)$. Firstly, we discuss some theorems about the complexity of calculating the relative fractional packing number and its hardness of approximation. Moreover, we will calculate this number exactly when both graphs $G$ and $H$ are cycles. Finally, we will generalize our previous technique to get some bounds on this number.

\subsection{Hardness Results}

\begin{theorem}
The problem of computing $\alpha^{*}(G|H)$ is $\mathcal{NP}$-Hard.
\end{theorem}

\begin{proof}
We will reduce the problem of finding the independence number of a graph to this problem, therefore showing that this problem is $\mathcal{NP}$-Hard.

    Imagine graph $G$ is a single vertex graph. In this case, we would have:
\begin{equation*}
    \alpha^{*}(G|H) = \max\limits_{W} \frac{\alpha(G \boxtimes W)}{\alpha(H \boxtimes W)} = \max\limits_{W} \frac{\alpha(W)}{\alpha(H \boxtimes W)} = \frac{1}{\inf\limits_{W} \frac{\alpha(H \boxtimes W)}{\alpha(W)}} = \frac{1}{\alpha(H)}
\end{equation*}
Where the last equality is because $\alpha(H \boxtimes W) \geq \alpha(H) \times \alpha(W)$.
\end{proof}

One can easily verify that the above reduction is, in fact, an approximation-preserving reduction. Therefore, since the problem of finding the independence number is Poly-APX-hard, the fractional packing number is also Poly-APX-hard.

\begin{remark}
The relative fractional packing number can't be approximated to a constant factor in polynomial time unless $\mathcal{P} = \mathcal{NP}$.
\end{remark}
\subsection{Relative Fractional Packing Number for Two Cycles}
In this section, we are going to calculate the fractional packing number when both graphs $G$ and $H$ are cycles. But first, we are going to mention a lemma that is proved in \cite{alipour2023relative}, which helps us in these calculations.
\begin{lemma}[Alipour and Gohari \cite{alipour2023relative}]
    \label{lem:vertrans}
    If graph $G$ is vertex-transitive with $n$ vertices, then for any arbitrary graph $H$, we have:
    \begin{equation*}
        \alpha^{*}(G|H) = \frac{n}{\alpha(G^{c} \boxtimes H)}
    \end{equation*}
\end{lemma} 

Since cycles are vertex-transitive graphs, we apply this lemma to calculate their relative fractional packing number.

Before stating the theorem, we are going to define the projection of an independence set to a vertex as follows:
\begin{definition}
    \label{def:proj}
    For an independence set $I$ of the graph $G \boxtimes H$ and a vertex $v \in G$, we define the projection of $I$ to $v$ as the set $\lbrace u \in H , | , (v,u) \in I \rbrace$. The definition is similar when projecting to a vertex in $H$.
\end{definition}
\begin{theorem}
    For two numbers $n, m \geq 3$, let $f(n,m)$ be the following function:
    \[ f(n,m)  = 
    \begin{cases} 
      \frac{n}{m} & \text{m is even} \\
      \frac{n}{m-1} & \text{n is even and m} \\
      \frac{n}{m} & \text{n and m are odd and n $\leq$ m} \\
      \frac{n}{m-1} & \text{n and m are odd and m $<$ n}
   \end{cases} \]
   Now, for any two cycles $C_n$ and $C_m$, we have $\alpha^{*}(C_n, C_m) = f(n,m)$.
\end{theorem}
\begin{proof}
    Since even cycles are perfect graphs, if the first case happens, then we have:
    \[
    \alpha^{*}(C_n, C_m) = \frac{n}{\alpha(C_n^c \boxtimes C_m)} = \frac{n}{\alpha(C_n^c) \times \alpha(C_m)}
    \]
    now $\alpha(C_n^c) = 2$ and $\alpha(C_m) = \frac{m}{2}$ therefore,
    \[
    \frac{n}{\alpha(C_n^c) \times \alpha(C_m)} = \frac{n}{m}
    \]
    
    The second case goes similarly, if $C_n$ is an even cycle, then it would be a perfect graph and by the perfect graph theorem \cite{LOVASZ197295}, $C_n^c$ will also be a perfect graph so the calculations goes as follows:
    \[
    \alpha^{*}(C_n, C_m) = \frac{n}{\alpha(C_n^c \boxtimes C_m)} = \frac{n}{\alpha(C_n^c) \times \alpha(C_m)}
    \]
    Now, this time, since $H$ is odd, $\alpha(C_m) = \frac{m-1}{2}$. Therefore,
    \[
    \frac{n}{\alpha(C_n^c) \times \alpha(C_m)} = \frac{n}{m-1}
    \]
    
    For the two latter cases, imagine that $u$ is an arbitrary vertex in $C_m$. Then, the size of the projection of any maximum independent set of $C_n^c \boxtimes C_m$ to $u$ is at most two because the largest possible independent set in $C_n^c$ has a size of two.

    Now, if for a vertex $u$, we have $|A_u| = 2$, then the set $A_u$ should contain two adjacent vertices in $C_n$. However, since two adjacent vertices in $C_n$ form a dominating set in $C_n^c$, the neighbors of $u$ in $G_m$ have empty projections. Therefore, the average projection size is less than or equal to one, meaning that $\alpha(C_n^c \boxtimes C_m) \leq m$. Additionally, it is obvious that $\alpha(C_n^c \boxtimes C_m) \geq \alpha(C_n^c) \times \alpha(C_m) \geq m-1$. So, we only have to check which of these two possibilities happens.

    When $n$ and $m$ are both odd and $n \leq m$, then we can have the following independent set of $C_n^c \boxtimes C_m$ with a size of $m$. Imagine vertices of $C_n$ are $u_1, \dots, u_n$ and vertices of $C_m$ are $v_1, \dots, v_m$ in order. Now, the claimed independent set for $C_n^c \boxtimes C_m$ is as follows:
    \begin{equation*}
        A = \lbrace (u_1, v_1), \dots, (u_n, v_n), (u_1, v_{n+1}), (u_2, v_{n+2}), (u_1, v_{n+3}), \dots, (u_2, v_{m}) \rbrace
    \end{equation*}
    You can easily verify that this is a valid independent set, which means $\alpha(C_n^c \boxtimes C_m) = m$. Therefore, $\alpha^{*}(C_n, C_m) = \frac{n}{m}$.

    For the final case, we show that it is impossible to have an independent set for $C_n^c \boxtimes C_m$ with a size of $m$.
    Suppose such an independent set $A$ exists. Based on previous reasoning, for any vertex $v \in C_m$, the size of the projection of $A$ to $v$ is exactly one. We denote this projection by the set $A_v$.
    Additionally, for any two adjacent vertices $v_i$ and $v_{i+1}$ in $C_m$, because their projections should be disconnected, if $A_{v_i}$ contains the vertex $u_j$, then $A_{v_{i+1}}$ either contains $u_{j+1}$ or $u_{j-1}$.
    
    Now, we label the edges of $C_m$ with the following rule: we label the edge between $v_i$ and $v_{i+1}$ in $C_m$ with a $+1$ if $A_{v_{i+1}}$ contains $u_{j+1}$, and otherwise, we put a $-1$ on that edge. If we take the summation of all the numbers on the edges of $C_m$, we should get $0$ because we are looping from a number back to itself in a cycle. However, this summation consists of $m$ $+1$ or $-1$ modulo $n$. Basic number theory tells us that this summation can't be equal to $0$ modulo $n$ when $m < n$ and both $n$ and $m$ are odd. This implies that $\alpha(C_n^c \boxtimes C_m) = m-1$, which means in this case $\alpha^{*}(C_n, C_m) = \frac{n}{m-1}$.
\end{proof}

\section{Necessary and Sufficient Conditions for Relative Fractional Packing Number to Be Bellow One}
\label{sec:3}
As shown in \cite{alipour2023relative}, the relative fractional packing number can serve as an upper bound on the ratio of many different graph properties. In fact, they showed the following theorem:
\begin{theorem}[Alipour and Gohari \cite{alipour2023relative}]
    Let $X(G)$ be any of the following graph properties: independence number, zero-error Shannon capacity, fractional packing number, or Lovász number of a graph $G$. Then we have the following:
\[
\frac{X(G)}{X(H)} \leq \alpha^{*}(G|H)
\]
\end{theorem}
The above theorem mentions an important usage for the relative fractional packing number. If we find a necessary and sufficient condition on graphs $G$ and $H$ such that $\alpha^{*}(G|H) \leq 1$, then according to the theorem, we have also found a sufficient condition for $X(G) \leq X(H)$ for all choices of $X$.

In this section, we propose a conjecture for such a condition and provide proofs for its correctness for some famous families of graphs.
\begin{definition}
    \label{def:expand}
    For a graph $G$, we define $Expand(G)$ as the set of all graphs that can be obtained by a sequence of the following three operations on $G$ iteratively.
    \begin{enumerate}
        \item Remove a vertex $v$ from $G$.
        \item add a new edge to $G$.
        \item replace a vertex $v$ by a clique of a arbitrarily size.
    \end{enumerate}
    The latter operation just puts a clique of size $k$ instead of vertex $v$ and connects these $k$ new vertices to all neighbors of $v$.
    
\begin{figure}[hbt!]
  \centering
  \subfloat{\includegraphics[width=0.4\textwidth]{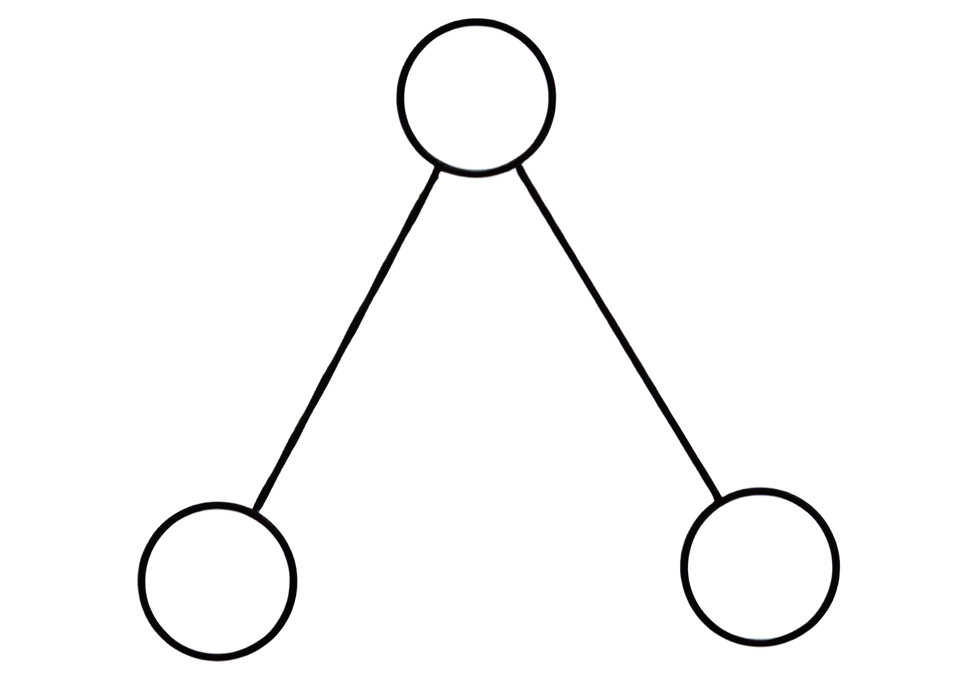}}
  \subfloat{\includegraphics[width=0.4\textwidth]{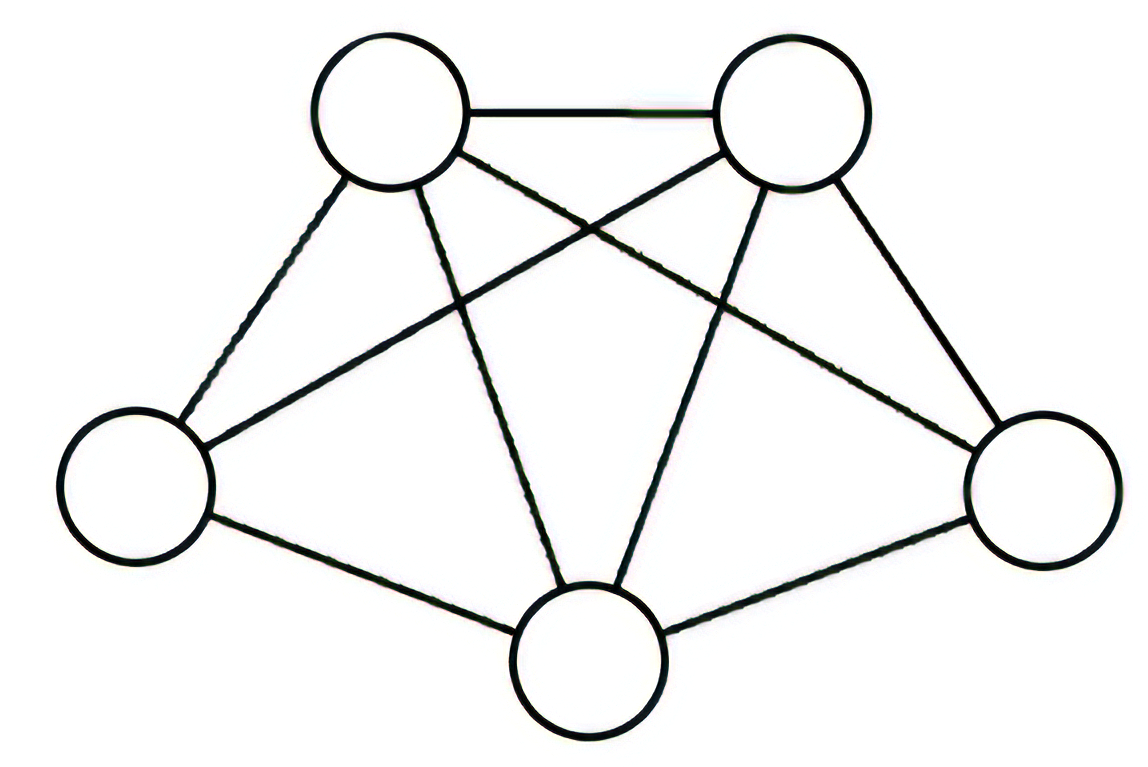}}
  \caption{The graph $P_3$ before and after replacing the middle vertex with a clique of size $3$ using the third operation.}
\end{figure}
    
    You can trivially verify that all the above operations, when applied to a graph $G$, will not increase $\alpha(G \boxtimes W)$ for any choice of graph $W$. Therefore, if we have two graphs $G$ and $H$ such that $G \in Expand(H)$, then we can conclude that $\alpha(G \boxtimes W) \leq \alpha(H \boxtimes W)$ for any graph $W$. This implies that $\alpha^{*}(G|H) \leq 1$.
\end{definition}
\begin{remark}
\label{mrk:merge}
    Note that by the above operations, one can merge (contract) any two vertices $v$ and $u$ in $G$. To achieve this, we only need to delete one of the vertices, for instance $v$, by applying the first operation. Then, we add new edges between $u$ and the neighbors of $v$ by applying the second operation. This allows us to merge the two vertices $v$ and $u$ into a single vertex.
\end{remark}
\begin{theorem}
    \label{th:triv}
    From the definition of expand, it is trivial to observe that if $G \in Expand(H)$, then $\alpha^{*}(G|H) \leq 1$.
\end{theorem}
It's time to mention the main conjecture.
\begin{conjecture}
    \label{conj:below1}
    $\alpha^{*}(G|H) \leq 1$ if and only if $G \in expand(H)$.
\end{conjecture}
Now, we will validate this conjecture when $G$ belongs to certain specific classes of graphs.
\\
Firstly, we examine the case where $G$ is a cycle.
\begin{theorem}
    \label{th:cycle}
    If $G$ is a cycle and $H$ is an arbitrary graph, then $\alpha^{*}(G|H) \leq 1$ if and only if $G \in Expand(H)$.
\end{theorem}
\begin{proof}
    According to Theorem \ref{th:triv}, one direction of the theorem is true. We will now assume $\alpha^{*}(G|H) \leq 1$ and prove that $G \in \text{Expand}(H)$.

Since $G$ is a cycle, it is also vertex-transitive. Therefore, by Lemma \ref{lem:vertrans}, we deduce:
\[
\alpha^{*}(G|H) = \frac{n_G}{\alpha(G^{c} \boxtimes H)}
\]
So $\alpha^{*}(G|H) \leq 1$ means that $n_G \leq \alpha(G^{c} \boxtimes H)$.

Recall the definition of projection from Definition \ref{def:proj}. For any vertex $u$ in $G$ or $H$, we name the projection from the maximum independent set of $G^{c} \boxtimes H$ to $u$ as $A_u$. Now we have:
\[
\alpha(G^{c} \boxtimes H) = \sum\limits_{u \in G} |A_u| \geq n_G
\]

We split the rest of the proof into two cases.
\\
    \textbf{Case 1  $\forall u \in G \, \, |A_u| \geq 1$} :
    
    For any vertex $v$ of $H$, $A_v$ is an independent set in $G^c$. Therefore, it forms a clique in $G$. Since $G$ is a cycle, this clique can be either empty, a single vertex, or a single edge. To construct $G$ from $H$ by applying the expansion operations, we first replace every vertex $v$ in $H$ with a clique of size $|A_v|$ to obtain $H'$. We can imagine the vertices of these cliques labeled by the members of $A_v$. If $A_v$ is empty, we simply delete the vertex $v$ from $H$.
    
    If two vertices in $H'$ labeled by $u_i$ and $u_j$ are connected by an edge, then either for a vertex $v$ in $H$ we have $u_i, u_j \in A_v$, or there exist two connected vertices $v_i, v_j \in H$ such that $u_i \in A_{v_i}$ and $u_j \in A_{v_j}$. Since these projections were from an independent set of $G^{c} \boxtimes H$, it follows that in both cases, $u_i$ must be disconnected from $u_j$ in $G^{c}$. Consequently, there is an edge between them in $G$.

    By the above observation, we can deduce that if there is an edge between two vertices labeled by $u_i$ and $u_j$ in $H'$, then there is an edge between $u_i$ and $u_j$ in $G$.
    
    Finally, using Remark \ref{mrk:merge}, we can merge all the vertices in $H'$ that have the same label to obtain $H''$. By the assumption of this case and the statement above, $H''$ is a spanning subgraph of $G$. Now, by adding the missing edges of $G$ to $H''$, we can finish the construction of $G$. Thus, the theorem is proved for this case.
    \\
    \textbf{Case 2 $\exists u \in G \, \, |A_u| = 0$} :
    Imagine such a vertex $u$ exists. Let's order the vertices of $G$ clockwise from $u_1$ to $u_{n_G}$, where we set $u = u_1$. Now, consider the following independent sets in $G$:
    \begin{equation*}
        S = \lbrace u_2, u_4, \dots, u_{n_G -1} \rbrace \quad T = \lbrace u_3, u_5, \dots, u_{n_G} \rbrace
    \end{equation*}
    Since $\sum\limits_{u \in G} |A_u| \geq n_G$ and by symmetry, we can assume $\sum\limits_{u \in S} |A_u| \geq \ \lceil \frac{n_G}{2} \rceil$. Now, similar to the previous case, we try to construct $G$ from $H$ by applying the expansion operations.
    
    Firstly, delete every vertex $v$ from $H$ such that $A_v \cap S = \emptyset$ to obtain $H'$. This is because $S$ was an independent set in $G$, so it forms a clique in $G^c$. Since the projections are coming from an independent set $G^c \boxtimes H$, all the remaining vertices in $H'$ must be disconnected from each other, and thus $H'$ is an independent set.
    
     Due to our assumption about the set $S$, $H'$ has more vertices than $\lceil \frac{n_G}{2} \rceil$. To obtain $H''$, we replace $\lfloor \frac{n_G}{2} \rfloor$ of these vertices with cliques of size two (edges). As a result, $H''$ consists of $\lfloor \frac{n_G}{2} \rfloor$ disjoint edges (and one single vertex if $n_G$ is odd). Hence, $H''$ is a spanning subgraph of $G$.

     Now, by adding the missing edges of $H''$ to it, we complete the construction of $G$. Thus, the theorem is also proved for this case.
    \end{proof}

    It is not hard to see that the proof in \textbf{Case 1} was completely independent from the fact that $G$ is a cycle. Therefore, if one can prove that the second case won't happen for a certain family of graphs, they have indeed proved the correctness of Conjecture \ref{conj:below1} for that certain family of graphs.
     
     Finally, we are going to investigate the correctness of Conjecture \ref{conj:below1} when $G$ is a perfect graph. We will apply a similar technique of using the projections to construct $G$ by applying expansion operations to $H$.

     But first, we are going to mention a couple of well-known theorems regarding perfect graphs.
    \begin{theorem}[Rosenfeld \cite{rosenfeld1967problem}]
        If $G$ is a perfect graph, then for any arbitrary graph $H$ we have:
        \begin{equation*}
            \alpha(G \boxtimes H) = \alpha(G) \times \alpha(H)
        \end{equation*}
    \end{theorem}
    \begin{theorem}[Lovász \cite{LOVASZ197295}]
        A graph $G$ is perfect if and only if $G^c$ is perfect.
    \end{theorem}
    \noindent
    The above theorem is know as the perfect graph theorem.
    \begin{theorem}[Alipour and Gohari \cite{alipour2023relative}]
        \label{th:perfectrelative}
        If $G$ is a perfect graph and $H$ is an arbitrary graph, then $\alpha^{*}(G|H) = \frac{\alpha(G)}{\alpha(H)}$.
    \end{theorem}
    \begin{theorem}
    \label{th:perfect}
    If $G$ is a perfect graph and $H$ is an arbitrary graph, then $\alpha^{*}(G|H) \leq 1$ if and only if $G \in \text{Expand}(H)$.
    \end{theorem}
    \begin{proof}
        Based on Theorem \ref{th:perfectrelative}, we know that if $G$ is a perfect graph, then $\alpha^{*}(G|H) = \frac{\alpha(G)}{\alpha(H)}$. Therefore, $\alpha^{*}(G|H) \leq 1$ implies that $\alpha(G) \leq \alpha(H)$. Additionally, it implies that for any arbitrary graph $W$, we have:

        \[
        \frac{\alpha(G \boxtimes W)}{\alpha(H \boxtimes W)} \leq 1
        \]
        
        By setting $W = G^c$, we obtain:
        
        \[
        \frac{n_G}{\alpha (H \boxtimes G^c)} \leq \frac{\alpha(G \boxtimes G^c)}{\alpha (H \boxtimes G^c)} \leq 1 \xrightarrow{\quad} n_G \leq \alpha(H \boxtimes G^c)
        \]
        
        Using the above inequalities, we will attempt to construct $G$ from $H$ by applying the expansion operations.
        
        By the Perfect Graph Theorem, we know that $G^c$ is also perfect. Therefore, we have $n_G \leq \alpha(H) \times \alpha(G^c)$. Consider the maximum independent set of $H \boxtimes G^c$ as the Cartesian product of the maximum independent sets of $H$ and $G^c$.
        
       We apply the technique used in the previous proof for cycles and replace vertices of $H$ with cliques of the size equal to the cardinality of their projection, resulting in the graph $H'$. Due to the selection of the maximum independent set, $H'$ is $\alpha(W)$ disjoint cliques, each of size $\alpha(G^c)$.
        
        Now we will prove a lemma that is critical for the remainder of our proof.

        \begin{lemma}
            If $G$ is a perfect graph, then its vertices can be partitioned into $\alpha(G)$ disjoint cliques.
        \end{lemma}
        \begin{proof}
            We know that if a graph $G$ is perfect, then for every induced subgraph of $G$, the maximum clique size is equal to the chromatic number. Moreover, by the perfect graph theorem, we know that if $G$ is perfect, then its complement $G^c$ is also perfect. So the maximum clique size of $G^c$, which is the same as the independence number of $G$, is equal to the chromatic number of $G^c$.
            
           If we view vertices of the same color as a set, then coloring a graph is just partitioning its vertices into disjoint independent sets. Therefore, a coloring of $G^c$ is equivalent to partitioning the vertices of $G$ into cliques. Because the chromatic number of $G^c$ is equal to $\alpha(G)$, this partition consists of $\alpha(G)$ cliques. Hence, the lemma is proven.
        \end{proof}
        Now, using the above theorem, we can observe that since cliques of $G^c$ are, in fact, independent sets of $G$, these cliques have a maximum size of $\alpha(G^c)$. This implies that we can remove some vertices from $H'$ to obtain $H''$ such that $H''$ precisely represents the clique partition of $G$.

       Therefore, $H''$ is a spanning subgraph of $G$, and by adding the remaining edges to $H''$, we will complete the construction of $G$ from $H$.
    \end{proof}
    \begin{corollary}
        From the above theorems, we can see that Conjecture \ref{conj:below1} is true for many famous classes of graphs such as bipartite graphs, cycles, stars, and friendship graphs.
    \end{corollary}

    \section{Open Problems and Future Directions} 
    We have demonstrated the truth of Conjecture \ref{conj:below1} for various classes of graphs. However, proving or disproving its correctness for all classes of graphs remains an intriguing future research direction. One interesting special case to investigate is when $G$ is vertex-transitive. Due to the symmetry in vertex-transitive graphs, one can reasonably assume that Case 2 of the proof for Theorem \ref{th:cycle} will not occur. Consequently, it appears probable that this conjecture holds true for this particular class of graphs.

    Next, we have established that the calculation of the relative fractional packing number is $\mathcal{NP}$-Hard. Nonetheless, we have successfully determined its exact value when both graphs $G$ and $H$ are cycles. One can try to generalize this proof and compute the relative fractional packing number for other scenarios, obtaining the relative fractional packing number for different classes of graphs. This exploration can provide insights into the relationship between the fractional packing number and the structure of various graph classes.

    Lastly, while investigating the properties of the relative fractional packing number, we were able to prove the following theorem:
    \begin{theorem}
        \label{thm:compact}
        For a graph $G$ and a fractional number $\alpha(G) \leq a \leq \alpha^{*}(G)$, there exists a graph $W_{a}$ such that:
        \[
        \frac{\alpha(G \boxtimes W_a)}{ \alpha (W_a)} = a
        \]
        
    \end{theorem}
    \noindent
    The Lovász number of a graph $G$ satisfies $\alpha(G) \leq \nu(G) \leq \alpha^{*}(G)$. By applying Theorem \ref{thm:compact}, we conclude that there exists a graph $W_{\nu}$ such that $\frac{\alpha(G \boxtimes W_{\nu})}{\alpha(W_{\nu})} = \nu(G)$. This observation suggests the potential existence of an algorithmic method to construct the graph $W_{\nu}$. Similar to the independence number, which is defined by an Integer Program, and the fractional packing number, which is defined by a Linear Program, the Lovász number is defined by a Semi Definite Program. Therefore, if such a property exists, it could potentially be deduced by analyzing the constraints and optimal solutions of the associated Semi Definite Program.

    \section{Acknowledgment}
    The author extends gratitude to Prof. Amin Gohari for his corrections, valuable comments, and introduction to the problem addressed in this paper.
    
\bibliographystyle{plain}
\bibliography{Relative_Fractional_Packing_Number_and_Its_Properties}

\newpage

\section*{Appendix A}
\subsection{Generalized Independence Number of a graph}

\begin{definition}
For a graph $G$ and an integer $k$, we define $\alpha_k(G)$ as the maximum number of vertices that can be selected from $G$ such that, from any clique in $G$, at most $k$ vertices are chosen. It is important to note that a vertex may be selected multiple times.
\end{definition}

\begin{remark}
You can see that if $k = 1$, the above definition is the same as the normal independence number of a graph.
\end{remark}

\begin{remark}
\label{remark:frac}
If we define $\alpha_{k}^{*}$ similar to the fractional packing number, then we have $\alpha_{k}^{*}(G) = k \alpha^{*}(G)$ because the linear program to compute $\alpha_{k}^{*}(G)$ is basically $k$ times the linear program to calculate $\alpha^{*}(G)$.
\end{remark}

Note that Remark \ref{remark:frac} does not hold for the integral version, meaning that we don't necessarily have equality in $\alpha_{k}(G) \geq k \alpha(G)$. A counterexample for this equality occurs when $k = 2$ and $G = C_5$. In this case, we have $2 \cdot \alpha(C_5) = 4$, but $\alpha_{2}(C_5) = 5$.

\begin{theorem}
The generalized independence number is superadditive, meaning that for a graph $G$ and two integral numbers $k_1$ and $k_2$, we have $\alpha_{k_1}(G) + \alpha_{k_2}(G) \leq \alpha_{k_1 + k_2}(G)$.
\end{theorem}

\begin{proof}
It is easy to see because if $x$ is a solution to the integer program for finding $\alpha_{k_1}(G)$ and $y$ is a solution to the integer program for finding $\alpha_{k_2}(G)$, then $x+y$ is a solution to the integer program for finding $\alpha_{k_1 + k_2}(G)$.
\end{proof}

Now, we can show that we can calculate the fractional packing number using the generalized version of the independence number.
\begin{theorem}
\label{thm:ineq}
We have $\alpha_{k}(G \boxtimes W) \leq \alpha^{*}(G) \alpha_{k}(W)$.
\end{theorem}

\begin{proof}
The relaxed linear program for calculating the fractional packing number is given by:
\begin{equation}
\label{LP:genfrac}
\begin{array}{lll}
\text{maximize } & \sum\limits_{u \in V} x_{u} \\
\text{subject to } & \sum\limits_{u \in C} x_{u} \leq 1 \quad \forall \, \, Clique \, \, C \\
            & x \geq 0
\end{array}
\end{equation}
Let $A$ be a maximum $k$-generalized independent set in $G \boxtimes W$. For a vertex $u \in G$, we define $A_u$ as the projection of $A$ onto vertex $u$, i.e., $A_u = {v \in W \,|\, (u,v) \in A}$. It can be easily seen that $A_u$ is still a $k$-generalized independent set, and for a set of vertices in a clique $C$ in $G$, the union of their projections is also a $k$-generalized independent set.

Now, let us set $x_v = \frac{|A_u|}{\alpha_{k}(W)}$ for all vertices $v$ in $W$. By the above arguments, we can conclude that $x$ is a feasible solution for LP \ref{LP:genfrac}. Hence, we have:
\[
\sum\limits_{u \in V} x_{u} = \sum\limits_{u \in V} \frac{|A_u|}{\alpha_{k}(W)} = \frac{\alpha_{k}(G \boxtimes W)}{\alpha_{k}(W)} \leq \alpha^{*}(G)
\]
Thus, we have shown that $\alpha_{k}(G \boxtimes W) \leq \alpha^{*}(G) \alpha_{k}(W)$, as desired.
\end{proof}
\begin{remark}
\label{remark:enqu}
If we set $W$ to be a graph with a single vertex, then the above inequality says that $\alpha_{k}(G) \leq k \alpha^{*}(G)$.
\end{remark}

\begin{theorem}
$\sup\limits_{W} \frac{\alpha_{k}(G \boxtimes W)}{\alpha_{k}(W)} = \alpha^{*}(G)$.
\end{theorem}

\begin{proof}
Based on Theorem \ref{thm:ineq}, we have one side of the inequality. For the other side, we need to show that there exists a graph $W$ for which equality holds. Let $x^{*}$ be the optimal solution to LP \ref{LP:genfrac}, and let $N$ be a large integer such that for any vertex $i \in G$, the number $n_i := x^{*}_{i}N$ is an integer.

Consider the graph $W = G^{c}(n_1, n_2, \dots, n_{|V|})$ where each vertex $u_i$ is repeated $n_i$ times. We claim that this graph maximizes the expression $\frac{\alpha_{k}(G \boxtimes W)}{\alpha_{k}(W)}$, and equality holds for this graph.

To find $\alpha_{k}(G \boxtimes W)$, we can pick vertex $u$ $n_u$ times. Since for any clique $C$, we have:
\[
\sum\limits_{u \in C} n_{u} = N \sum\limits_{u \in C} x^{*} \leq N
\]
This new set is a feasible solution for $\alpha_{k}(G \boxtimes W)$, and its value is the following:
\[\sum\limits_{u \in V} n_{u} = N \sum\limits_{u \in V} x^{*} = N \alpha^{*}(G)\]
By Remark \ref{remark:enqu}, this is the optimal solution and we have shown that $\sup\limits_{W} \frac{\alpha_{k}(G \boxtimes W)}{\alpha_{k}(W)} = \alpha^{*}(G)$.
\end{proof}

\begin{theorem}
\label{thm:oneN}
For any graph $G$, there exists an integer $k$ such that $\alpha_{k}(G) = k \alpha^{*}(G)$.
\end{theorem}

\begin{proof}
Define $x^{*}$ and $N$ as in the previous theorem. By the previous inequalities it is trivial that we have $\alpha_{N}(G) = N \alpha^{*}(G)$. Hence, we have shown that for any graph $G$, there exists an integer $k$ such that $\alpha_{k}(G) = k \alpha^{*}(G)$.
\end{proof}
\begin{theorem}
For a collection of graphs $G_1, \dots , G_k$, there exists an integer $k$ such that for every $i$, $\alpha_{k}(G_i) = k \alpha^{*}(G_i)$.
\end{theorem}

\begin{proof}
Using Theorem \ref{thm:oneN} and the superadditivity of the generalized independent number, we can see that if we set $k = N_1 \times \dots \times N_k$, where $N_i$ is a large integer such that multiplying the optimal solution of the LP for graph $G_i$ by $N_i$ gives an integer answer, then we have the desired property.
\end{proof}

\begin{theorem}
For any graph $G$, we have $\lim\limits_{n \to \infty} \frac{\alpha_{n}(G)}{n} = \alpha^{*}(G)$.
\end{theorem}

\begin{proof}
First, set $N$ to be the large number for which Theorem \ref{thm:oneN} holds. Now, consider a number $M = Nq + r$, where $q, r \in \mathbb{N}$ and $r < N$. By the superadditivity of the generalized independence number, we have $(M-N) \alpha^{*}(G) \leq qN \alpha^{*}(G) \leq qN \alpha^{*}(G) + \alpha_{r}(G) \leq \alpha_{M}(G)$. Additionally, using Remark \ref{remark:enqu}, we have $\alpha_{M}(G) \leq M \alpha^{*}(G)$.

Now, we can conclude that:
\begin{equation*}
\lim\limits_{n \to \infty} \frac{(n-N) \alpha^{*}(G)}{n} \leq \lim\limits_{n \to \infty} \frac{\alpha_{n}(G)}{n} \leq \lim\limits_{n \to \infty} \frac{n \alpha^{*}(G)}{n}
\end{equation*}
Since $N$ is a constant, we can simplify further and state that $\lim\limits_{n \to \infty} \frac{\alpha_{n}(G)}{n} = \alpha^{*}(G)$. Thus, the theorem is proved.

The similarity of the above theorem with the definition of the fractional chromatic number is really interesting. If we take the dual of the integer program used to calculate the $k$-generalized independent number of a graph $G$, it corresponds to covering the vertices of $G$ with cliques such that each vertex is covered by at least $k$ cliques. If we assign different colors to each vertex in different cliques, it becomes a $k$-fold coloring of the vertices of $G^c$, which is one of the ways we define the fractional chromatic number.
\end{proof}
\end{document}